\documentclass[a4paper]{amsart} 
\usepackage{amssymb,times, amscd,amsmath,amsthm, xypic}
\author{Shoji Yokura$^{(*)}$}
\address
{Department of Mathematics and Computer Science, 
Faculty of Science, 
Kagoshima University, 21-35 Korimoto 1-chome, Kagoshima 890-0065, Japan}
\email {yokura@sci.kagoshima-u.ac.jp}
\title
{Oriented bivariant theories, I}
\thanks {(*) Partially supported by Grant-in-Aid for Scientific Research
(No. 19540094), the Ministry of Education, Culture, Sports, Science and Technology (MEXT), and JSPS Core-to-Core Program 18005, Japan \\
\quad \emph{keywords} : Fulton--MacPherson's bivariant theory, (co)bordism, Chern classes, Grothendieck--Riemann--Roch, $K$-theory}

\begin{document} 
\numberwithin{equation}{section}
\newtheorem{thm}[equation]{Theorem}
\newtheorem{pro}[equation]{Proposition}
\newtheorem{prob}[equation]{Problem}
\newtheorem{cor}[equation]{Corollary}
\newtheorem{lem}[equation]{Lemma}
\theoremstyle{definition}
\newtheorem{ex}[equation]{Example}
\newtheorem{defn}[equation]{Definition}
\newtheorem{rem}[equation]{Remark}
\renewcommand{\rmdefault}{ptm}
\def\alp{\alpha}
\def\be{\beta}
\def\jeden{1\hskip-3.5pt1}
\def\om{\omega}
\def\bigstar{\mathbf{\star}}
\def\ep{\epsilon}
\def\vep{\varepsilon}
\def\Om{\Omega}
\def\la{\lambda}
\def\La{\Lambda}
\def\si{\sigma}
\def\Si{\Sigma}
\def\Cal{\mathcal}
\def\ga{\gamma}
\def\Ga{\Gamma}
\def\de{\delta}
\def\De{\Delta}
\def\bF{\mathbb{F}}
\def\bH{\mathbb H}
\def\bPH{\mathbb {PH}}
\def \bB{\mathbb B}
\def \bA{\mathbb A}
\def \bOB{\mathbb {OB}}
\def \bM{\mathbb M}
\def \bOM{\mathbb {OM}}
\def \calB{\mathcal B}
\def \bK{\mathbb K}
\def \bG{\mathbf G}
\def \bL{\mathbf L}
\def\bN{\mathbb N}
\def\bR{\mathbb R}
\def\bP{\mathbb P}
\def\bZ{\mathbb Z}
\def\bC{\mathbb  C}
\def \bQ{\mathbb Q}
\def\op{\operatorname}

\begin{abstract} In 1981 W. Fulton and R. MacPherson introduced the notion of bivariant theory ({\bf BT}), which is a sophisticated unification of  covariant theories and contravariant theories. This is for the study of singular spaces. In 2001 M. Levine and F. Morel introduced the notion of algebraic cobordism, which is a universal oriented Borel--Moore functor with products ({\bf OBMF}) of geometric type, in an attempt to understand better V. Voevodsky's (higher) algebraic cobordism. In this paper we introduce a notion of oriented bivariant theory ({\bf OBT}), a special case of which is nothing but the oriented Borel--Moore functor with products. The present paper is a first one of the series to try to understand Levine--Morel's algebraic cobordism from a bivariant-theoretical viewpoint, and its first step is to introduce {\bf OBT} as a unification of {\bf BT} and {\bf OBMF}.
\end{abstract}

\maketitle

\section{Introduction}\label{intro} 

William Fulton and Robert MacPherson have introduced the notion of bivariant theory as a {\it categorical framework for the study of singular spaces}, which is the title of their AMS Memoir book \cite{Fulton-MacPherson} (see also Fulton's book
\cite{Fulton-book}). The main objective of \cite {Fulton-MacPherson} is bivariant-theoretic Riemann--Roch's or bivariant analogues of various theorems of Grothendieck--Riemann--Roch type. \\

Vladimir Voevodsky has introduced algebraic cobordism (now called {\it higher algebraic cobordism}), which was used in his proof of Milnor's conjecture \cite{Voevodsky}. Daniel Quillen introduced the notion of {\it (complex ) oriented cohomology theory} on the category of differential manifolds \cite{Quillen} and this notion can be formally extended to the category of smooth schemes in algebraic geometry. Marc Levine and Fabien Morel constructed {\it a universal oriented cohomology theory}, which  they also call {\it algebraic cobordism}, and have investigated furthermore (see \cite{Levine-ICM}, \cite{Levine-survey}, \cite{LM-CRAS1}, \cite{LM-CRAS2}, \cite{LM-book} and see also  \cite{Merkurjev} for a condensed review). Recently  M. Levine and R. Pandharipande \cite{LP} gave another equivalent construction of the algebraic cobodism via what they call ``double point degeneration" and they found a nice application of the algebraic cobordism in the Donaldson--Thomas theory of 3-folds. \\

In this paper we extend Fulton--MacPherson's bivariant theory to what we call {\it an oriented bivariant theory} for a general category, not just for a geometric category of, say, complex algebraic varieties, schemes, etc. In most interesting cases bivariant theories such as \emph{bivariant homology theory, bivariant Chow group theory, bivariant algebraic K-theory and bivariant topological K-theory} are already oriented bivariant theories. We show that even in this general category there exists {\it a universal oriented bivariant theory}, whose special case gives rise to {\it a universal oriented Borel--Moore functor with products}. Levine--Morel's algebraic cobordism requires more geometrical conditions.
Indeed, they call algebraic cobordism  {\it a universal oriented Borel--Moore functor with products \underline {of geometric type}} \cite{LM-book}. In a second paper \cite{Yokura-bac} we will deal with an oriented bivariant theory of geometric type. In \cite {Yokura-bbg} we apply our approach to (co)bordism groups. \\

One purpose of this paper is to bring Fulton--MacPherson's Bivariant Theory to the attention of people working on algebraic cobordism and/or subjects related to it.

\section {Fulton--MacPherson's Bivariant Theory}\label{FM-BT}

We make a quick review of Fulton--MacPherson's bivariant theory \cite {Fulton-MacPherson}. 

Let $\Cal V$ be a category which has a final object $pt$ and on which the fiber product or fiber square is well-defined. Also we consider a class of maps, called ``confined maps" (e.g., proper maps, projective maps, in algebraic geometry), which are closed under composition and base change and contain all the identity maps, and a class of fiber squares, called ``independent squares" (or ``confined squares", e.g., ``Tor-independent" in algebraic geometry, a fiber square with some extra conditions required on morphisms of the square), which satisfy the following:

(i) if the two inside squares in  

$$\CD
X''@> {h'} >> X' @> {g'} >> X \\
@VV {f''}V @VV {f'}V @VV {f}V\\
Y''@>> {h} > Y' @>> {g} > Y \endCD
$$

or

$$\CD
X' @>> {h''} > X \\
@V {f'}VV @VV {f}V\\
Y' @>> {h'} > Y \\
@V {g'}VV @VV {g}V \\
Z'  @>> {h} > Z \endCD
$$

are independent, then the outside square is also independent,

(ii) any square of the following forms are independent:
$$
\xymatrix{X \ar[d]_{f} \ar[r]^{\op {id}_X}&  X \ar[d]^f & & X \ar[d]_{\op {id}_X} \ar[r]^f & Y \ar[d]^{\op {id}_Y} \\
Y \ar[r]_{\op {id}_X}  & Y && X \ar[r]_f & Y}
$$
where $f:X \to Y$ is \emph{any} morphism. \\

A bivariant theory $\bB$ on a category $\Cal V$ with values in the category of graded abelian groups is an assignment to each morphism
$$ X  \xrightarrow{f} Y$$
in the category $\Cal V$ a graded abelian group (in most cases we ignore the grading )

$$\bB(X  \xrightarrow{f} Y)$$
which is equipped with the following three basic operations. The $i$-th component of $\bB(X  \xrightarrow{f} Y)$, $i \in \bZ$, is denoted by $\bB^i(X  \xrightarrow{f} Y)$.\\

{\bf Product operations}: For morphisms $f: X \to Y$ and $g: Y
\to Z$, the product operation
$$\bullet: \bB^i( X  \xrightarrow{f}  Y) \otimes \bB^j( Y  \xrightarrow{g}  Z) \to
\bB^{i+j}( X  \xrightarrow{gf}  Z)$$
is  defined.

{\bf Pushforward operations}: For morphisms $f: X \to Y$
and $g: Y \to Z$ with $f$ \emph {confined}, the pushforward operation
$$f_*: \bB^i( X  \xrightarrow{gf} Z) \to \bB^i( Y  \xrightarrow{g}  Z) $$
is  defined.

{\bf Pullback operations}: For an \emph{independent} square
$$\CD
X' @> g' >> X \\
@V f' VV @VV f V\\
Y' @>> g > Y, \endCD
$$
the pullback operation
$$g^* : \bB^i( X  \xrightarrow{f} Y) \to \bB^i( X'  \xrightarrow{f'} Y') $$
is  defined.

These three operations are required to satisfy the seven compatibility axioms (see \cite [Part I, \S 2.2]{Fulton-MacPherson} for details):\\

(B-1) product is associative, 

(B-2) pushforward is functorial,

 (B-3) pullback is functorial, 

(B-4) product and pushforward commute,

 (B-5) product and pullback commute, 

(B-6) pushforward and pullback commute, and 

(B-7) projection formula. \\

We also assume that $\bB$ has units:

{\bf Units}: $\bB$ has units, i.e., there is an element $1_X \in \bB^0( X  \xrightarrow{\op {id}_X} X)$ such that $\alp \bullet 1_X = \alp$ for all morphisms $W \to X$, all $\alp \in \bB(W \to X)$; such that $1_X \bullet \beta = \beta $ for all morphisms $X \to Y$, all $\beta \in \bB(X \to Y)$; and such that $g^*1_X = 1_{X'}$ for all $g: X' \to X$. \\

Let $\bB, \bB'$ be two bivariant theories on a category $\Cal V$. Then
a {\it Grothendieck transformation} from $\bB$ to $\bB'$
$$\ga : \bB \to \bB'$$
is a collection of homomorphisms
$$\bB(X \to Y) \to \bB'(X \to Y) $$
for a morphism $X \to Y$ in the category $\Cal V$, which preserves the above three basic operations: 

\hskip1cm (i) \quad $\ga (\alp \bullet_{\bB} \be) = \ga (\alp) \bullet _{\bB'} \ga (\be)$, 

\hskip1cm  (ii) \quad $\ga(f_{*}\alp) = f_*\ga (\alp)$, and 

\hskip1cm  (iii) \quad $\ga (g^* \alp) = g^* \ga (\alp)$. \\

For more details of interesting geometric and/or topological examples of bivariant theories (e.g., bivariant theory of constructible functions, bivariant homology theory, bivariant $K$-theory, etc.,) and Grothendieck transformations among bivariant theories, see \cite {Fulton-MacPherson}. In this paper we treat with bivariant theories more abstractly from a general viewpoint. \\

A bivariant theory unifies both a covariant theory and a contravariant theory in the following sense:

$\bB_*(X):= \bB(X \to pt)$ becomes a covariant functor for {\it confined}  morphisms and 

$\bB^*(X) := \bB(X  \xrightarrow{id}  X)$ becomes a contravariant functor for {\it any} morphisms. 

\noindent
A Grothendieck transformation $\ga: \bB \to \bB'$ induces natural transformations $\ga_*: \bB_* \to \bB_*'$ and $\ga^*: \bB^* \to {\bB'}^*$.

As to the grading, $\bB_i(X):= \bB^{-i}(X  \xrightarrow{id}  X)$ and
$\bB^j(X):= \bB^j(X  \xrightarrow{id}  X)$.
  
In the rest of the paper we assume that our bivariant theories are \emph{commutative} (see \cite {Fulton-MacPherson}, \S 2.2), i.e., if whenever both

$$
\xymatrix{W \ar[d]_{f'} \ar[r]^{g'}&  X \ar[d]^f && W \ar[d]_{g'} \ar[r]^{f'}& Y \ar[d]^{g} \\
Y \ar[r]_{g}  & Z && X \ar[r]_f & Z}
$$
are independent squares, then for 
$\alp \in \bB(X  \xrightarrow {f} Z)$ and $\be \in \bB(Y  \xrightarrow {g} Z)$
$$g^*(\alp) \bullet \be = f^*(\be) \bullet \alp.$$
(Note: if $g^*(\alp) \bullet \be = (-1)^{\op {deg}(\alp) \op{deg}(\be)} f^*(\be) \bullet \alp$, then it is called \emph{skew-commutative}.)

\begin{defn}\label{canonical}(\cite{Fulton-MacPherson}, Part I, \S 2.6.2 Definition) Let $\Cal S$ be a class of maps in $\Cal V$, which is closed under compositions and containing all identity maps. Suppose that to each $f: X \to Y$ in $\Cal S$ there is assigned an element
$\theta(f) \in \bB(X  \xrightarrow {f} Y)$ satisfying that
\begin{enumerate}
\item [(i)] $\theta (g \circ f) = \theta(f) \bullet \theta(g)$ for all $f:X \to Y$, $g: Y \to Z \in \Cal S$ and

\item [(ii)] $\theta(\op {id}_X) = 1_X $ for all $X$ with $1_X \in \bB^*(X):= B(X  \xrightarrow{\op {id}_X} X)$ the unit element.
\end{enumerate}
Then $\theta(f)$ is called a {\it canonical orientation} of $f$.
\end{defn} 

Note that such a canonical orientation makes the covariant functor $\bB_*(X)$ a contravariant functor for morphisms in $\Cal S$, and also makes the contravariant functor $\bB^*$ a covariant functor for morphisms in $\Cal C \cap \Cal S$: Indeed, 

(*) for a morphism $f: X \to Y \in \Cal S$ and the canonical orientation $\theta$ on $\Cal S$ the following {\it Gysin homomorphism}
$$f^!: \bB_*(Y) \to \bB_*(X) \quad \text {defined by} \quad  f^!(\alp) :=\theta(f) \bullet \alp$$
 is {\it contravariantly functorial}. And 

(**) for a fiber square (which is an independent square by hypothesis)
$$\CD
X @> f >> Y \\
@V {\op {id}_X} VV @VV {\op {id}_Y}V\\
X @>> f > Y, \endCD
$$
where $f \in \Cal C \cap  \Cal S$, the following {\it Gysin homomorphism}
$$f_!: \bB^*(X) \to \bB^*(Y) \quad \text {defined by} \quad
f_!(\alp) := f_*(\alp \bullet \theta (f))$$
is {\it covariantly functorial}.
The notation should carry the information of $\Cal S$ and the canonical orientation $\theta$, but it will be usually omitted if it is not necessary to be mentioned. Note that the above conditions (i) and (ii) of Definition (\ref{canonical}) are certainly necessary for the above Gysin homomorphisms to be functorial. \\

\begin{defn} (i) Let $\Cal S$ be another class of maps called ``specialized maps" (e.g., smooth maps in algebraic geometry) in $\Cal V$ , which is closed under composition, closed under base change and containing all identity maps. Let $\bB$ be a bivariant theory. If $\Cal S$ has canonical orientations in $\bB$, then we say that $\Cal S$ is canonically $\bB$-oriented and an element of $\Cal S$ is called a canonically $\bB$-oriented morphism. (Of course $\Cal S$ is also a class of confined maps, but since we consider the above extra condition of $\bB$-orientation  on $\Cal S$, we give a different name to $\Cal S$.)

(ii) Let $\Cal S$ be as in (i). Let $\bB$ be a bivariant theory and $\Cal S$ be canonically $\bB$-oriented. Furthermore, if the orientation $\theta$ on $\Cal S$ satisfies that for an independent square with $f \in \Cal S$
$$
\CD
X' @> g' >> X\\
@Vf'VV   @VV f V \\
Y' @>> g > Y
\endCD
$$
the following condition holds 
\[\theta (f') = g^* \theta (f),\]
(which means that the orientation $\theta$ preserves the pullback operation), then we call $\theta$ a {\it nice canonical orientation} and say that $\Cal S$ is {\it nice canonically $\bB$-oriented} and an element of $\Cal S$ is called {\it a nice canonically $\bB$-oriented morphism} .
\end{defn}

In the following proposition we deal with \emph {cross product} (\cite{Fulton-MacPherson}, Part I, \S 2.4 External products). For that we need the assumption that all the four small squares in the following big diagrams are independent (hence any square is independent):
$$\CD
X \times Y @> {f \times Id_Y} >> X' \times Y @> {p \times Id_Y} >>  Y \\
@V {Id_X \times g }VV @VV {Id_{X'} \times g}V @VV {g}V\\
X \times Y' @> {f \times Id_{Y'}} >> X' \times Y' @> {p \times Id_{Y'}} >> Y' \\
@V {Id_X \times q }VV @VV {Id_{X'} \times q}V @VV {q}V\\
X  @>> {f} > X'  @>> {p}> pt .\endCD
$$

\begin {pro}  Let $\bB$ be a bivariant theory and let $\Cal S$ be as above. 

(1) Define the natural exterior product 
$$\times : \bB(X \xrightarrow {} pt ) \times \bB(Y  \xrightarrow{\pi_Y}  pt) \to \bB( X \times Y \to pt)$$
by 
$$\alp \times \be := \pi_Y^*\alp \bullet \be.$$
Then the covariant functor $\bB_*$ for confined morphisms and the contravariant functor $\bB_*$ for morphisms in $\Cal S$ are both compatible with the exterior product, i.e., for confined morphisms $f:X \to X'$, $g:Y \to Y'$, 
$$(f \times g)_*(\alp \times \be) = f_*\alp \times g_*\be$$
and for morphisms $f:X \to X'$, $g:Y \to Y'$ in $\Cal S$,
$$(f \times g)^!(\alp' \times \be') = f^!\alp' \times g^!\be'.$$

(2) Similarly, define the natural exterior product 
$$\times : \bB(X \xrightarrow {\op {id}_X} X) \times \bB(Y  \xrightarrow{\op {id}_Y}  Y) \to \bB( X \times Y  \xrightarrow {\op {id}_{X \times X}} X \times X)$$
by 
$$\alp \times \be := {p_1}^*\alp \bullet {p_2}^*\be$$
where $p_1:X \times Y \to X$ and $p_2:X \times Y \to Y$ be the projections.

Then the contravariant functor $\bB^*$ for any morphisms and the covariant functor $\bB^*$ for morphisms in $\Cal C \cap \Cal S$ are both compatible with the exterior product, i.e., for any morphisms $f:X \to X'$, $g:Y \to Y'$, 
$$(f \times g)^*(\alp \times \be) = f^*\alp \times g^*\be$$
and for morphisms $f:X \to X'$, $g:Y \to Y'$ in $\Cal C \cap \Cal S$,
$$(f \times g)_!(\alp' \times \be') = f_!\alp' \times g_!\be'.$$
\end{pro}

\begin {proof} The proof is tedious, using several axioms of the bivariant theory. For the sake of completeness we give a proof.  But, we give a proof for only (1) and a proof for (2) is left for the reader. 

For morphisms $f: X \to X'$ and $g: Y \to Y'$, consider  the above big commutative diagrams.
The proof of $(f \times g)_*(\alp \times \be) = f_*\alp \times g_*\be$ goes as follows:

\begin {align*}
(f \times g)_*(\alp \times \be) & = (f \times g)_*\Bigl ( (qg)^* \alp \bullet \be \Bigr ) \quad \text {(by definition)} \\
& = (f \times g)_*\Bigl ( \left (g^* \left (q^*\alp \right ) \right ) \bullet \be \Bigr )  \\
& = (Id_{X'} \times f)_*\biggl ( \left (f \times Id_Y \right )_* \Bigl (g^* \left (q^*\alp \right ) \bullet \be \Bigr ) \biggr) \\
& = (Id_{X'} \times f)_*\biggl ( \left (f \times Id_Y \right )_* \left (g^* \left (q^*\alp \right ) \right ) \bullet \be \biggr) \quad \text {(by (B-4))} \\
& = (Id_{X'} \times f)_*\biggl ( g^* \Bigl ( \left (f \times Id_{Y'} \right )_*\left (q^*\alp \right ) \Bigr  ) \biggr ) \bullet \be \quad \text {(by (B-6))} \\
& = (Id_{X'} \times f)_*\Bigl ( g^* \left (q^* \left (f_*\alp \right ) \right ) \bullet \be \Bigr) \quad \text {(by (B-6))} \\
& = q^* \left (f_*\alp \right ) \bullet g_* \be \quad \text {(by (B-7))} \\
& = f_*\alp \times g_* \be \quad \text {(by definition)} 
\end{align*}

Next we show $(f \times g)^!(\alp' \times \be') = f^!\alp' \times g^!\be'.$
For this, first we observe that 
$$(f \times g)^! := (\theta (f)  \times \theta (g) ) \bullet .$$
On one hand we have that 
\begin {align*}
(f \times g)^! (\alp' \times \be') & = (\theta (f)  \times \theta (g) ) \bullet (q^*\alp' \bullet \be') \quad \text {(by definition)} \\
& = (\theta (f)  \times \theta (g) ) \bullet q^*\alp' \bullet \be' \\
& = \Bigl (Id_{X'} \times f)^* (Id_{Y'} \times q)^* \Bigr) (\theta (f) ) \bullet (p \times Id_{Y'})^*(\theta (g) ) \bullet q^*\alp' \bullet \be' \
\end{align*}

On the other hand we have that  
\begin {align*}
f^!\alp'  \times g^!\be' & = (\theta (f)  \bullet \alp') \times (\theta (g)  \bullet \be') \quad \text {(by definition)} \\
& = \Bigl ((qg)^* (\theta (f) \bullet \alp') \Bigr ) \bullet (\theta (g) \bullet \be') \\
& = \biggl ( g^* \Bigl (q^* (\theta (f) \bullet \alp') \Bigr ) \biggr ) \bullet (\theta (g)  \bullet \be') \\
& = g^* \Bigl ( (Id_{Y'} \times q)^*\alp_f \bullet q^*\alp' \Bigr ) \bullet \theta (g)  \bullet \be'\\
& = \Bigl (Id_{X'} \times f)^* (Id_{Y'} \times q)^* \Bigr) (\theta (f) ) \bullet g^*(q^*\alp') \bullet \theta (g)  
\bullet \be' \\
& = \Bigl (Id_{X'} \times f)^* (Id_{Y'} \times q)^* \Bigr) (\theta (f) ) \bullet (p \times Id_{Y'})^*(\theta (g) ) \bullet q^*\alp' \bullet \be' \
\end{align*}

The last equality follows from the commutativity of the bivariant theory.
Thus we get the above equality. 
\end {proof}

Here we remark the following fact about the covariant and contravariant functors $\bB_*$ and $\bB^*$, which will be needed in later sections. They are \emph {almost} what Levine and Morel call \emph{Borel--Moore functor with products} in \cite{LM-book} (see also \cite{Merkurjev}); namely they do not necessarily have the \emph {additivity property}, which is explained below after the proposition.

\begin{pro} Let the situation be as above.

(1-i) for confined morphisms $f:X \to Y$, the pushforward homomorphisms
$$f_*: \bB_*(X) \to \bB_*(Y)$$
 are covariantly functorial,

(1-ii) for morphisms in $\Cal S$, i.e., for nice canonical $\bB$-orientable morphisms
$f:X \to Y$, the Gysin (pullback) homomorphisms 
$$f^!: \bB_*(Y) \to \bB_*(X)$$
are contravariantly functorial,

(1-iii) for an independent square
$$
\CD
X' @> g' >> X \\
@V f' VV @VV f V\\
Y' @>> g > Y 
 \endCD
$$
with $g \in \Cal C$ and $f \in \Cal S$, the following diagram commutes:
$$
\CD
\bB_*(Y') @> {f'}^! >> \bB_*(X') \\
@V g_* VV @VV {g'}_* V\\
\bB_*(Y) @>> f^! > \bB_*(X), 
 \endCD
$$

(1-iv) the pushforward homomorphisms
$f_*: \bB_*(X) \to \bB_*(Y)$ for confined morphisms and the Gysin (pullback) homomorphisms $f^!: \bB_*(Y) \to \bB_*(X)$ for  morphisms in $\Cal S$ are both compatible with the exterior products
$$ \times :\bB_*(X) \otimes \bB_*(Y) \to \bB_*(X \times Y).$$ 
 
(2-i) for any morphisms $f:X \to Y$, the pullback homomorphisms
$$f^*: \bB^*(Y) \to \bB^*(Y)$$
 are contravariantly functorial,

(2-ii) for confined and specialized morphisms in $\Cal C \cap \Cal S$, i.e., for confined and nice canonical $\bB$-orientable morphisms
$f:X \to Y$, the Gysin (pushforward) homomorphisms 
$$f_!: \bB^*(X) \to \bB^*(Y)$$
are covariantly functorial,

(2-iii) for an independent square
$$
\CD
X' @> g' >> X \\
@V f' VV @VV f V\\
Y' @>> g > Y 
 \endCD
$$
with $g \in \Cal C \cap \Cal S$, the following diagram commutes:
$$
\CD
\bB^*(Y') @> {f'}^*>> \bB^*(X') \\
@V g_! VV @VV {g'}_! V\\
\bB^*(Y) @>> f^! > \bB^*(X), 
 \endCD
$$

(2-iv) the pullback homomorphisms
$f^*: \bB^*(Y) \to \bB^*(Y)$ for any morphisms $f:X \to Y$ and the Gysin (pushforward) homomorphisms $f_!: \bB^*(X) \to \bB^*(Y)$ for  confined and specialized morphisms in $\Cal C \cap \Cal S$ are both compatible with the exterior products
$$ \times :\bB^*(X) \otimes \bB^*(Y) \to \bB^*(X \times Y).$$\
\end{pro}

\begin{rem} As mentioned above, what Levine and Morel call {\it Borel--Moore functor with products} requires the following additivity property:e.g., if $H_*$ is the usual homology theory, for the disjoint union $X \coprod Y$ of spaces
$$H_*(X \coprod Y) = H_*(X) \oplus H_*(Y).$$
If we want a bivariant theory to have such an additivity property, we need more requirements on the category. We assume that

(1) our category is closed under taking the coproduct  $\coprod$, 

(2) the morphisms $i_X:X \to X \coprod Y$ and $i_Y: Y \to X \coprod Y$ are confined and nice canonical $\bB$-orientable and strongly orientable in the sense of Fulton--MacPherson \cite[2.6 Orientations]{Fulton-MacPherson}, 

(3) $f: X \coprod Y \to Z$ is confined if and only if $f|_X := f\circ i_X :X \to Z$ and $f|_Y := f\circ i_Y :Y \to Z$ are confined,

(4) our category satisfies that any fiber square with $f \in \Cal C$ 
$$\CD
P' @>  >> P\\
@V f' VV @VV f V\\
Q' @>> > Q \endCD
$$
is \emph {independent}. This assumption shall be provisionally called \emph {``$\Cal C$-independence"}.

Under these assumpitons, the additivity property for our bivariant theory means that the following homomorphism is an isomorphism:
$${i_X}_* \oplus {i_Y}_*: \bB(X   \xrightarrow {f|_X} Z) \oplus  \bB(Y   \xrightarrow {f|_Y} Z)  \xrightarrow {\cong } \bB(X \coprod Y  \xrightarrow {f} Z).$$
The special cases imply the following additivity formulas:for the coproduct  $X \coprod Y$, 
$$\bB_*(X \coprod Y) \cong  \bB_*(X) \oplus \bB_*(Y),$$
$$\bB^*(X \coprod Y) \cong  \bB^*(X) \oplus \bB^*(Y).$$
This additivity property is not so important, but when we need this additivity property and want to emphasize it, we call such a bivariant theory an {\it additive bivariant theory}. 
\end{rem}

\section{A Universal Bivariant Theory}

The following theorem is about {\it the existence of the universal bivariant theory} in a class of bivariant theories defined on a category $\Cal V$ which is equipped with a class $\Cal C$ of confined morphisms, a class of independent squares and a class $\Cal S$ of specialized morphisms. 

\begin{thm}\label{UBT} Let  $\Cal V$ be a category equipped with a class $\Cal C$ of confined morphisms, a class of independent squares and a class $\Cal S$ of specialized maps.  We define 
$$\bM^{\Cal C} _{\Cal S}(X  \xrightarrow{f}  Y)$$
to be the free abelian group generated by the set of isomorphism classes of confined morphisms $h: W \to X$  such that the composite of  $h$ and $f$ is a specialized map:
$$h \in \Cal C \quad \text {and} \quad f \circ h: W \to Y \in \Cal S.$$

(1) The assignment $\bM^{\Cal C} _{\Cal S}$ is a bivariant theory if the three operations are defined as follows:

{\bf Product operations}: For morphisms $f: X \to Y$ and $g: Y
\to Z$, the product operation
$$\bullet: \bM^{\Cal C} _{\Cal S} ( X  \xrightarrow{f}  Y) \otimes \bM^{\Cal C} _{\Cal S} ( Y  \xrightarrow{g}  Z) \to
\bM^{\Cal C} _{\Cal S} ( X  \xrightarrow{gf}  Z)$$
is  defined by
$$\left (\sum_V m_V[V \xrightarrow{h_V}  X] \right ) \bullet \left (\sum_W n_W[W  \xrightarrow{k_W}  Y] \right ) := \sum_{V, W} m_V n_W [V'  \xrightarrow{h_V \circ k_W''}  X], $$
where we consider the following fiber squares
$$\CD
V' @> {h_V'} >> X' @> {f'} >> W \\
@V {k_W''}VV @V {k_W'}VV @V {k_W}VV\\
V@>> {h_V} > X @>> {f} > Y @>> {g} > Z .\endCD
$$

{\bf Pushforward operations}: For morphisms $f: X \to Y$
and $g: Y \to Z$ with $f$ confined, the pushforward operation
$$f_*: \bM^{\Cal C} _{\Cal S} ( X  \xrightarrow{gf} Z) \to \bM^{\Cal C} _{\Cal S} ( Y  \xrightarrow{g}  Z) $$
is  defined by
$$f_*\left (\sum_Vn_V[V \xrightarrow{h_V}  X] \right) := \sum _Vn_V[V  \xrightarrow{f \circ h_V}  Y].$$

{\bf Pullback operations}: For an independent square
$$\CD
X' @> g' >> X \\
@V f' VV @VV f V\\
Y' @>> g > Y, \endCD
$$
the pullback operation
$$g^* : \bM^{\Cal C} _{\Cal S} ( X  \xrightarrow{f} Y) \to \bM^{\Cal C} _{\Cal S}( X'  \xrightarrow{f'} Y') $$
is  defined by
$$g^*\left (\sum_V n_V[V  \xrightarrow{h_V}  X] \right):=  \sum_V n_V[V'  \xrightarrow{h_V'}  X'],$$
where we consider the following fiber squares:
$$\CD
V' @> g'' >> V \\
@V {h_V'} VV @VV {h_V}V\\
X' @> g' >> X \\
@V f' VV @VV f V\\
Y' @>> g > Y. \endCD
$$

(2)  Let $\Cal {BT}$ be a class of bivariant theories $\bB$ on the same category $\Cal V$ with a class $\Cal C$ of confined morphisms, a class of independent squares and a class $\Cal S$ of specialized maps. We also assume that our category satisfies the ``$\Cal C$-independence"  defined in the previous section. Let $\Cal S$ be nice canonically $\bB$-oriented for any bivariant theory $\bB \in \Cal {BT}$. Then, for each bivariant theory $\bB \in \Cal {BT}$ there exists a unique Grothendieck transformation
$$\ga_{\bB} : \bM^{\Cal C} _{\Cal S} \to \bB$$
such that for a specialized morphism $f: X \to Y \in \Cal S$ the homomorphism
$\ga_{\bB} : \bM^{\Cal C} _{\Cal S}(X  \xrightarrow{f}  Y) \to \bB(X  \xrightarrow{f}  Y)$
satisfies the normalization condition that $$\ga_{\bB}([X  \xrightarrow{\op {id}_X}  X]) = \theta_{\bB}(f).$$
\end{thm}

\begin{rem} (1) By the definition of $\bM^{\Cal C}_{\Cal S} $, the class $\Cal S$ is nice canonically $\bM^{\Cal C}_{\Cal S}$-oriented with the canonical orientation $\theta_{\bM^{\Cal C}_{\Cal S} } (X  \xrightarrow{f}  Y) := [X  \xrightarrow{id_X}  X]$ for $f \in \Cal S$. 

(2) The product operation $\bullet: \bM^{\Cal C} _{\Cal S} ( X  \xrightarrow{f}  Y) \otimes \bM^{\Cal C} _{\Cal S} ( Y  \xrightarrow{g}  Z) \to
\bM^{\Cal C} _{\Cal S} ( X  \xrightarrow{gf}  Z)$ can also be interpreted as follows. The free abelian group $\Cal M(X)$ generated by the set of isomorphism classes of confined morphisms $h_V:V \to X$ is a commutative ring by the fiber product
$$[V_1  \xrightarrow{h_1}  X]\cup [V_2  \xrightarrow{h_
2}  X]:= [V_1 \times_X V_2  \xrightarrow{h_1 \times _X h_2}  X].$$
For a confined morphism $f:X \to Y$ we have the pushforward homomorphism $f_*:\Cal M(X) \to \Cal M(Y)$ and for \emph{any} morphism $f:X \to Y$ we have the pullback homomorphism $f^*:\Cal M(Y) \to \Cal M(X)$. Then the product operation is nothing but 
$$[V  \xrightarrow{h_V}  X] \bullet [W  \xrightarrow{h_W}  Y] = [V  \xrightarrow{h_V}  X] \cup f^*( [W  \xrightarrow{h_W}  Y]). $$
But in our case we need to chase the morphisms involved, so we just stick to this presentation.
\end{rem}

Now we go on to the proof of Theorem \ref{UBT}:
\begin{proof} For (1), we have to show that the three bivariant
operations are well-defined, but we show only the well-definedness of the bivariant product and the other two are clear. To show that these three operations satisfy the seven axioms (B-1) --- (B-7) is left for the reader. 

Let $[V  \xrightarrow{h_V} X] \in \bM^{\Cal C}_{\Cal S} ( X  \xrightarrow{f} Y)$ and $[W  \xrightarrow{k_W} Y] \in \bM^{\Cal C}_{\Cal S} ( Y  \xrightarrow{g} Z)$; thus $h_V:V \to X$ is confined and the composite $f \circ h_V:V \to Y$ is in $\Cal S$, and also  $k_W:W \to Y$ is confined and the composite $g \circ k_W:W \to Z$ is in $\Cal S$. By definition we have
$$[V  \xrightarrow{h_V} X] \bullet [W  \xrightarrow{k_W} Y] 
= [V'  \xrightarrow{h_V \circ k_W''} X].$$
We want to show that $[V'  \xrightarrow{h_V \circ k_W''} X] \in \bM^{\Cal C}_{\Cal S} ( X  \xrightarrow{g \circ f} Z)$, i.e., 
$$ (g \circ f ) \circ \left (h_V \circ k_W'' \right) \in \Cal S.$$
From the fiber squares given in {\bf Product operations} above, we have 
$$(g \circ f ) \circ \left (h_V \circ k_W'' \right) =
(g \circ k_W) \circ  \left (f' \circ h_V' \right ).$$
$f' \circ h_V'$ is in $\Cal S$, because it is the pullback of $f \circ h_V$ and $f \circ h_V$ is in $\Cal S$ and $\Cal S$ is closed under base change by hypothesis. $g \circ k_W$ is in $\Cal S$ by hypothesis. Thus the composite $(g \circ k_W) \circ  \left (f' \circ h_V' \right )$ is also in $\Cal S$. Thus the bivariant product is well-defined. \\

For (2), first we show the uniqueness. Suppose that there exists a Grothendieck transformation 
$$\ga : \bM^{\Cal C} _{\Cal S}(X  \xrightarrow{f}  Y) \to \bB(X  \xrightarrow{f}  Y)$$
such that for any $f: X \to Y \in \Cal S$ the homomorphism
$\ga : \bM^{\Cal C} _{\Cal S}(X  \xrightarrow{f}  Y) \to \bB(X  \xrightarrow{f}  Y)$
satisfies that $\ga([X  \xrightarrow{\op {id}_X}  X]) = \theta_{\bB}(f).$ Note that for any $f: X \to Y \in \Cal S$, $[X  \xrightarrow{\op {id}_X}  X] \in \bM^{\Cal C} _{\Cal S}(X  \xrightarrow{f}  Y) $ is a nice canonical orientation, i.e., $\theta_{\bM^{\Cal C} _{\Cal S}}(f) = [X  \xrightarrow{\op {id}_X}  X].$

Let $h_V:V \to X$ be a confined map such that $f \circ h_V: V \to Y$ is in $\Cal S$. We have that $[V  \xrightarrow{h_V} X]= {h_V}_*[V  \xrightarrow{\op {id}_V} V]$, where $[V  \xrightarrow{\op {id}_V} V] \in \bM^{\Cal C} _{\Cal S}(V  \xrightarrow{f \circ h_V}  Y)$. Since $f \circ h_V \in \Cal S$ by hypothesis, it follows from the normalization that we get 
\begin{align*}
\ga([V  \xrightarrow{h_V} X]) & = \ga ({h_V}_*[V  \xrightarrow{\op {id}_V} V]) \\
&={h_V}_*\ga([V  \xrightarrow{\op {id}_V} V]) \\
& ={h_V}_*\theta_{\bB}(f \circ h_V).\
\end{align*}
Thus it is uniquely determined. 

The rest is to show that the assignment 
$$\ga : \bM^{\Cal C} _{\Cal S}(X  \xrightarrow{f}  Y) \to \bB(X  \xrightarrow{f}  Y)$$
defined by $\ga_{\bB}([V  \xrightarrow{h_V} X]) = {h_V}_*\theta_{\bB}(f \circ h_V)$ is well-defined and it is also a Grothendieck transformation, i.e., that it preserves the three bivariant operations.

(i) \underline {the well-definedness of the above assignment $\ga$}: namely, it does not depend on the choice of $h_V: V \to X$. So, let us choose another one $h_{V'} : V' \to X$, i.e., we have the following commutative diagram:
$$\CD
V' @>  {\cong }>> V\\
@V {h_{V'}} VV @VV {h_V} V\\
X@>> {id_X} > X. \endCD
$$
Since $h_V \in \Cal C$ and the diagram is a fiber square, it follows from the $\Cal C$-independence assumption that it is independent. Therefore the outer square of the following diagram is independent since the lower square is independent by hypothesis:
$$\CD
V' @>  {\cong }>> V\\
@V {h_{V'}} VV @VV {h_V} V\\
X@>> {id_X} > X \\
@V f VV @VV f V\\
Y'@>> {id_Y} > Y. \endCD
$$
Since $f \circ h_V \in \Cal S$, the outersquare is independent and $\Cal S$ is nice canonically $\bB$-oriented, we have
$$\theta (f \circ h_{V'}) = id_Y^* \theta (f \circ h_{V}).$$
Hence  
\begin {align*}
{h_{V'}}_* \theta (f \circ h_{V'})  & = {h_{V'}}_* (id_Y^* \theta (f \circ h_{V})) \\
& = id_Y^* ({h_{V}}_*\theta (f \circ h_{V}))  \qquad \text {(by (B-6) )}\\
 & ={h_{V}}_*\theta (f \circ h_{V}).
\end{align*}
Thus it does not depend on the choice of $h_V: V \to X$.

(ii) \underline {it preserves the product operation}: Letting the situation be as in (1), it suffices to show that
$$\ga_{\bB} \left ([V  \xrightarrow{h_V} X]\bullet [W  \xrightarrow{k_W} Y] \right )  = \ga_{\bB}([V  \xrightarrow{h_V} X]) \bullet \ga_{\bB}([W  \xrightarrow{k_W} Y]).
$$
Using the fiber squares given in {\bf  Product operations}, we have
\begin {align*}
& \ga_{\bB} \left ([V  \xrightarrow{h_V} X]\bullet [W  \xrightarrow{k_W} Y] \right ) \\
 & = \ga_{\bB}([V'  \xrightarrow{h_V \circ k_W''} X]) \quad \text {(by the definition)} \\
& = \left (h_V \circ {k_W''} \right) _* \theta_{\bB}(g \circ f \circ h_V \circ k_W'') \quad \text {(by the definition)} \\
& = {h_V}_*{k_W''}_* \theta_{\bB}(g \circ k_W \circ f' \circ h_V') \\
& = {h_V}_*{k_W''}_* \left ( \theta_{\bB}(f' \circ h_V') \bullet \theta_{\bB}(g \circ k_W) \right ) 
\end{align*}
Here we need the assumption of \emph {$\Cal C$-independence}.   In the fiber squares 
$$\CD
V' @> {h_V'} >> X' @> {f'} >> W \\
@V {k_W''}VV @V {k_W'}VV @V {k_W}VV\\
V@>> {h_V} > X @>> {f} > Y  \endCD
$$
$k_W: W \to Y$ is confined by the definition, hence the outer square is independent by this $\Cal C$-independence assumption. Therefore, since $f \circ h_V : V \to Y$ is in $\Cal S$, the above equality continues as follows:
\begin{align*}
& = {h_V}_*{k_W''}_*\left ({k_W}^{\bigstar} \theta_{\bB}((f \circ h_V) )\bullet \theta_{\bB}(g \circ k_W) \right ) \\
& = {h_V}_*\left (\theta_{\bB}(f \circ h_V) \bullet {k_W}_*\theta_{\bB}(g \circ k_W) \right ) \quad \text {(by (B-7) projection formula)} \\
& = {h_V}_*\theta_{\bB}(f \circ h_V) \bullet {k_W}_*\theta_{\bB}(g \circ k_W)  \quad \text {(by (B-4))}\\
& = \ga_{\bB}([V  \xrightarrow{h_V} X]) \bullet \ga_{\bB}([W  \xrightarrow{k_W} Y]). \
\end{align*}

(iii) \underline {it preserves the pushforward operation}: Consider $X  \xrightarrow{f} Y \xrightarrow{g} Z$ and a confined morphsim $h_V:V \to X$ such that the composite $g \circ f \circ h_V: V \to Y$ is in $\Cal S$.

\begin{align*}
\ga_{\bB}(f_*[V  \xrightarrow{h_V} X]) &= \ga_{\bB}([V  \xrightarrow{f \circ h_V} Y]) \\
& = (f \circ h_V)_{\bigstar}\theta_{\bB}(g \circ (f \circ h_V)) \\
& = f_* {h_V}_{\bigstar}\theta_{\bB}((g \circ f) \circ h_V) \\
& = f_* \ga_{\bB}([V  \xrightarrow{h_V} X])\
\end{align*}

(iv) \underline {it preserves the pullback operation}: Consider a confined morphsim $h_V:V \to X$ such that the composite $f \circ h_V: V \to Y$ is in $\Cal S$ and the fiber squares given in {\bf Pullback operations} above, we have

\begin{align*}
\ga_{\bB}(g^*[V  \xrightarrow{h_V} X]) &= \ga_{\bB}([V'  \xrightarrow{h_V'} X']) \\
& = {h_V'}_*\theta_{\bB}(f' \circ h_V') \\
& = {h_V'}_*g^*\theta_{\bB}(f \circ h_V) \\
& = g^*{h_V}_*\theta_{\bB}(f \circ h_V) \quad \text {(by (B-6))}\\
& = g^*\ga_{\bB}([V  \xrightarrow{h_V} X]).\
\end{align*}
This completes the proof of the theorem.
\end{proof}

Let $\Cal S$ be a class of specialized morphisms as above and let $\Cal S$ be canonically $\bB$-oriented for a bivariant theory $\bB$. If $\pi_X:X\to pt$ is in $\Cal S$, in which case we sometimes say, abusing words, that {\it $X$ is specialized}, then we have the Gysin homomorphism
$$ {\pi_X}^!: \bB_*(pt) \to \bB_*(X)$$
which, we recall, is defined to be
$${\pi_X}^!(\alp) = \theta_{\bB}(\pi_X) \bullet \alp.$$
In particular, if we let $1_{pt} \in \bB(pt)$ be the unit, then we have
$${\pi_X}^!(1_{pt}) = \theta_{\bB}(\pi_X) \bullet 1_{pt} = \theta_{\bB}(\pi_X) .$$
This element ${\pi_X}^!(1_{pt})  = \theta_{\bB}(\pi_X)$ is called {\it the fundamental ``class" of $X$ associated to the bivariant theory $\bB$} (cf. \cite{LM-book}, \cite{Merkurjev}), denoted by $[X]_{\bB}$.

\begin{cor} Let $\Cal {BT}$ be a class of additive bivariant theories $\bB$ on the same category $\Cal V$ with a class $\Cal C$ of confined morphisms, a class of independent squares  and a class $\Cal S$ of specialized maps. Let $\Cal S$ be nice canonically $\bB$-oriented for any bivariant theory $\bB \in \Cal {BT}$. We also assume that our category satisfies the ``$\Cal C$-independence". Then, for each bivariant theory $\bB \in \Cal {BT}$,

(1) there exists a unique natural transformation 
$${\gamma_{\bB}}_* : {\bM^{\Cal C} _{\Cal S}}_* \to \bB_*$$
such that if $\pi_X:X \to pt$ is in $\Cal S$ the homomorphism ${\gamma_{\bB}}_* : {\bM^{\Cal C} _{\Cal S}}_*(X) \to \bB_*(X)$ satisfies that
$${\gamma_{\bB}}_*[X  \xrightarrow{\op {id}_X} X] = {\pi_X}^!(1_{pt}) = [X]_{\bB},$$
and

(2) there exists a unique natural transformation 
$${\gamma_{\bB}}^* : {\bM^{\Cal C} _{\Cal S}}^* \to \bB^*$$
such that for any $X$ the homomorphism ${\gamma_{\bB}}^* : {\bM^{\Cal C} _{\Cal S}}^*(X) \to \bB^*(X)$ satisfies that
$${\gamma_{\bB}}^*[X  \xrightarrow{\op {id}_X} X] = 1_X \in \bB^*(X).$$
\end{cor}

\begin{ex}\label{BT-example} Here we recall some important examples of bivariant theories from \cite{Fulton-MacPherson}. In these examples,  in each category $\Cal V$ we let $\Cal C = \Cal Prop$ be the class of proper morphisms and $\Cal S = \Cal Sm$ be the class of smooth morphisms and any fiber square is independent. 

\noindent
{\bf NOTE}: If we \emph {do not require the universality} of $\bM^{\Cal C} _{\Cal S}$, then we can take other morphisms such as local complete intersection morphisms for $\Cal S$, and also we can consider other more restricted squares such as Tor-independent squares for independent squares.\\

(1) \underline {Bivariant theory of constructible functions $\bF$}: Let $\Cal V$ be the category of complex analytic or algebraic varieties. Then there is a unique Grothendieck transformation
$$\ga_{\bF} : \bM^{\Cal Prop} _{\Cal Sm} \to \bF$$
such that for $f: X \to Y \in \Cal Sm$ the homomorphism $\ga_{\bF}:\bM^{\Cal Prop} _{\Cal Sm}(X  \xrightarrow{f}  Y) \to \bF(X  \xrightarrow{f}  Y)$ satsifies the normalization condition:
$\ga_{\bF}([X  \xrightarrow{\op {id}_X}  X]) = 1_f = \jeden_X.$

We have unique natural transformations
$${\ga_{\bF}}_* : {\bM^{\Cal Prop} _{\Cal Sm}}_* \to \bF_* = F$$
such that for any smooth variety $X$, ${\ga_{\bF}}_*([X  \xrightarrow{\op {id}_X}  X]) = \jeden_X\in F(X)$
and
$${\ga_{\bF}}^* : {\bM^{\Cal Prop} _{\Cal Sm}}^* \to \bF^*$$
such that for \emph {any} variety $X$, ${\ga_{\bF}}^*([X  \xrightarrow{\op {id}_X}  X]) = \jeden_X \in F^*(X).$
Here $F^*(X)$ is the abelian group of locally constant functions on $X$. \\

(2) \underline{Bivariant homology theory $\bH$}: Let $\Cal V$ be the category of complex analytic or algebraic varieties. Then there is a unique Grothendieck transformation
$$\ga_{\bH} : \bM^{\Cal Prop} _{\Cal Sm} \to \bH$$
such that for $f: X \to Y \in \Cal Sm$ the homomorphism
$\ga_{\bH} : \bM^{\Cal Prop} _{\Cal Sm}(X  \xrightarrow{f}  Y) \to \bH(X  \xrightarrow{f}  Y)$
satisfies the normalization condition: $\ga_{\bH}([X  \xrightarrow{\op {id}_X}  X]) = U_f.$
For the construction of the canonical orientation $U_f$, see \cite {Fulton-MacPherson}, Part II, \S 1.3, or \cite{BFM}, \S IV.4.
In particular, we have unique natural transformation:
$${\ga_{\bH}}_* : {\bM^{\Cal Prop} _{\Cal Sm}}_* \to \bH_* = H^{\op {BM}}_*$$
such that for any smooth variety $X$, ${\ga_{\bH}}_*([X  \xrightarrow{\op {id}_X}  X]) = [X] \in H^{\op {BM}}_*(X)$
and
$${\ga_{\bH}}^* : {\bM^{\Cal Prop} _{\Cal Sm}}^* \to \bH^*= H^*$$
such that for \emph {any} variety $X$, ${\ga_{\bH}}^*([X  \xrightarrow{\op {id}_X}  X]) = 1 \in H^*(X).$
Here $H^{\op {BM}}_*(X)$ is the Borel--Moore homology group and $H^*(X)$ is the usual cohomology group. \\

(3) \underline {Bivariant Chow group theory (or Operational bivariant Chow group theory) $\bA$}:  Let $\Cal V$ be the category of schemes. Then there is a unique Grothendieck transformation
$$\ga_{\bA} : \bM^{\Cal Prop} _{\Cal Sm} \to \bA$$
such that for  $f: X \to Y \in \Cal Sm$ the homomorphism
$\ga_{\bA} : \bM^{\Cal Prop} _{\Cal Sm}(X  \xrightarrow{f}  Y) \to \bA(X  \xrightarrow{f}  Y)$
satisfies the normalization condition: $\ga_{\bA}([X  \xrightarrow{\op {id}_X}  X])  = [f].$

We have unique natural transformations
$${\ga_{\bA}}_* : {\bM^{\Cal Prop} _{\Cal Sm}}_* \to \bA_* = A_* \quad (\text {or} \quad CH_*)$$
such that for any smooth scheme $X$, ${\ga_{\bA}}_*([X  \xrightarrow{\op {id}_X}  X]) = [X] \in A_*(X)$
and
$${\ga_{\bA}}^* : {\bM^{\Cal Prop} _{\Cal Sm}}^* \to \bA^*= A^* \quad (\text {or} \quad CH^*)$$
such that for \emph {any} scheme $X$, ${\ga_{\bA}}^*([X  \xrightarrow{\op {id}_X}  X]) = 1 \in A^*(X).$
Here $A_* = CH_*$ is the Chow homology group and $A^* = CH^*$ is the Chow cohomology group (see \cite{Fulton-book}). \\

(4) \underline {Bivariant algebraic $K$-theory $\bK_{\op {alg}}$}: \, Let $\Cal V$ be the category of quasi-projective \, schemes.  Then there is a unique Grothendieck transformation
$$\ga_{\bK_{\op {alg}}} : \bM^{\Cal Prop} _{\Cal Sm} \to \bK_{\op {alg}}$$
such that for  $f: X \to Y \in \Cal Sm$ the homomorphism
$\ga_{\bK_{\op {alg}}} : \bM^{\Cal Prop} _{\Cal Sm}(X  \xrightarrow{f}  Y) \to \bK_{\op {alg}}(X  \xrightarrow{f}  Y)$
satisfies the normalization condition: $\ga_{\bK_{\op {alg}}}([X  \xrightarrow{\op {id}_X}  X]) =  \Cal O_f.$
For the canonical orientation $\Cal O_f$, see \cite {Fulton-MacPherson}, Part II, \S 1.2.
We have unique natural transformations
$${\ga_{\bK_{\op {alg}}}}_* : {\bM^{\Cal Prop} _{\Cal Sm}}_* \to {\bK_{\op {alg}}}_* = K_0^{\op {alg}}$$
such that for any smooth scheme $X$, ${\ga_{\bK_{\op {alg}}}}_*([X  \xrightarrow{\op {id}_X}  X]) = [\Cal O_X] \in K_0^{\op {alg}}(X)$
and
$${\ga_{\bK_{\op {alg}}}}^* : {\bM^{\Cal Prop} _{\Cal Sm}}^* \to \bK_{\op {alg}}^* = K^0_{\op {alg}}$$
such that for \emph{any} scheme $X$, ${\ga_{\bK_{\op {alg}}}}^*([X  \xrightarrow{\op {id}_X}  X]) = 1 \in K^0_{\op {alg}}(X).$
\\

(5) \underline {Bivariant topological $K$-theory $\bK_{\op {top}}$}:\, Let  $\Cal V$ be the category of quasi-projective schemes. Then there is a unique Grothendieck transformation
$$\ga_{\bK_{\op {top}}} : \bM^{\Cal Prop} _{\Cal Sm} \to \bK_{\op {top}}$$
such that for  $f: X \to Y \in \Cal Sm$ the homomorphism
$\ga_{\bK_{\op {top}}} : \bM^{\Cal Prop} _{\Cal Sm}(X  \xrightarrow{f}  Y) \to \bK_{\op {top}}(X  \xrightarrow{f}  Y)$
satisfies the normalization condition: $\ga_{\bK_{\op {top}}}([X  \xrightarrow{\op {id}_X}  X])  = \Lambda_f.$ 
For the construction of the canonical orientation $\Lambda_f$, see \cite{Fulton-MacPherson}, Part II, \S 1.3,  or \cite{BFM}, \S IV.4.
We have unique natural transformation
$${\ga_{\bK_{\op {top}}}}_* : {\bM^{\Cal Prop} _{\Cal Sm}}_* \to {\bK_{\op {top}}}_* = K_0^{\op {top}}$$
such that for any smooth scheme $X$, ${\ga_{\bK_{\op {top}}}}_*([X  \xrightarrow{\op {id}_X}  X]) = \{X\} \in K_0^{\op {top}}(X)$
and
$${\ga_{\bK_{\op {top}}}}^* : {\bM^{\Cal Prop} _{\Cal Sm}}^* \to {\bK_{\op {top}}}^* = K^0_{\op {top}}$$
such that for \emph{any} variety $X$, ${\ga_{\bK_{\op {top}}}}^*([X  \xrightarrow{\op {id}_X}  X]) = 1 \in K^0_{\op {top}}(X).$ For more details of the topological $K$-theory, see \cite{BFM2}.
\\
\end{ex}

\begin{cor} (A na\"\i ve ``motivic" bivariant characteristic class) Let $c\ell:K^0 \to H^*(\quad)\otimes R$ be a multiplicative characteristic class of complex vector bundles with a suitable coefficients $R$. Then there exists a unique Grothendieck transformation
$$\ga_{\bH}^{c\ell} : \bM^{\Cal Prop} _{\Cal Sm} \to \bH(\quad) \otimes R$$
satisfying the normalization condition that for $f: X \to Y \in \Cal Sm$ 
$$\ga_{\bH}^{c\ell}([X  \xrightarrow{\op {id}_X}  X]) = c\ell(T_f) \bullet U_f.$$
Here $T_f$ is the relative tangent bundle of the smooth morphism $f$.
\end{cor}

\begin{proof} It suffices to point out that the multiplicative characteristic cohomology class \, $c\ell(T_f) \in H^*(X) = \bH( X  \xrightarrow{\op {id}_X}  X)$ satisfies the following properties:

\begin{enumerate}
\item
For smooth morphisms $f: X \to Y$ and $g: Y \to Z$ we have
$$ cl(T_f) \bullet g^*cl(T_g) = cl(T_f) \cup cl(g^*T_g) = cl(T_{g\circ f}).$$
\item
$c\ell(T_{\op {id}_X}) = 1 \in H^*(X).$
\item
for any fiber square with $f: X \to Y \in \Cal Sm$
$$\CD
X' @> g' >> X\\
@V f' VV @VV f V\\
Y' @>> g> Y \endCD
$$
$$g^*c\ell(T_f) = c\ell(T_{f'}).$$
\end{enumerate} 
\end{proof} 

\begin{cor}(A na\"\i ve ``motivic"  characteristic class of singular varieties) Let $c\ell:K^0 \to H^*(\quad)\otimes R$ be a multiplicative characteristic class of complex vector bundles with a suitable coefficients $R$. Then there exists a unique natural transformation 
$$c\ell_*: {\bM^{\Cal Prop} _{\Cal Sm}}_* \to H_*(\quad) \otimes R$$
such that for a smooth variety 
$$c\ell_*([V  \xrightarrow{\op {id}_V} V]) = c\ell(TV) \cap [V].$$
Here $TV$ is the tangent bundle of $V$.\\
\end{cor}

\begin{rem} In the case of algebraic varieties the covariant theory ${\bM^{\Cal Prop} _{\Cal Sm}}_* $ is used in \cite{BSY1} (see also \cite{BSY2}) (in which it is denoted by $Iso^{pr}(sm/X)$)  and there is a canonical natural transformation from ${\bM^{\Cal Prop} _{\Cal Sm}}_* $ to the covariant functor $K_0(\Cal Var / \quad )$ of relative Grothendieck group of varieties:
$$can: {\bM^{\Cal Prop} _{\Cal Sm}}_*(X)   \to K_0(\Cal Var / X)$$
which is surjective for any variety $X$ and F. Bittner \cite{Bi} proved that its kernel is described by the so-called ``blow-up relation" (see also \cite {BSY1}). 
$$
\xymatrix{ & {\bM^{\Cal Prop} _{\Cal Sm}}_*(X)  \ar[dl]_{can}  \ar[dr]^{c\ell_*} \\
 K_0(\Cal Var / X)\ar[rr]_{\natural} & & H_*(X) \otimes R}
$$
Certain restrictions are required on multiplicative characteristic classes $c\ell$ so that a homomorphism $\natural: K_0(\Cal Var / X) \to H_*(X) \otimes R$ exists and the above triangle becomes  commutative. For more details of such a homomorphism $\natural: K_0(\Cal Var / X) \to H_*(X) \otimes R$, see \cite{BSY1} (see also \cite{SY}). \

\end{rem}

Further discussions on ``motivic" bivariant characteristic classes will be done in a different paper.\\

\begin{rem}(Riemann--Roch Theorems)
We have the following commutative diagrams:
$$
\xymatrix{ & \bM^{\Cal Prop} _{\Cal Sm} \ar[dl]_{\ga_{\bK_{\op {alg}}}}  \ar[dr]^{\ga_{\bK_{\op {top}}}} & && \bM^{\Cal Prop} _{\Cal Sm} \ar[dl]_{\ga_{\bK_{\op {top}}}}   \ar[dr]^{\ga_{\bH}^{td}}  \\
\bK_{\op {alg}} \ar[rr]_{\alp} & & \bK_{\op {top}} & \bK_{\op {top}} \ar[rr]_{ch} & & \bH_{\bQ} }
$$

(i) $\alp: \bK_{\op {alg}} \to \bK_{\op {top}}$ is a Grothendieck transformation such that for a $\ell.c.i$ morphism $f: X \to Y$, $\alp (\Cal O_f) = \Lambda_f$. 

(ii) $ch: \bK_{\op {top}} \to \bH_{\bQ}$ is a Grothendieck transformation such that for an $\ell.c.i$ morphism $f: X \to Y$, $\alp (\Lambda_f) = td(T_f) \bullet U_f$,
where $td(T_f)$ is the total Todd cohomology class of the  relative tangent bundle  $T_f$ of the smooth morphism $f$. The composite $ch \circ \alp : \bK_{\op {alg}} \to \bH_{\bQ}$ is a bivariant version of Baum--Fulton--MacPherson's Riemann--Roch $\tau: K_0 \to {H_*}_{\bQ}$ (see \cite{BFM} and \cite{Fulton-book}). Thus one could say that $ch \circ \alp : \bK_{\op {alg}} \to \bH_{\bQ}$ and $ch: \bK_{\op {top}} \to \bH_{\bQ}$ are  realizations of a ``motivic" one: $\ga_{\bH}^{td} : \bM^{\Cal Prop} _{\Cal Sm} \to \bH_{\bQ}$.
\end{rem}

\section {Oriented Bivariant Theories}\label{OBT}

Levine--Morel's algebraic cobordism is the universal one among the so-called  \emph {oriented} Borel--Moore functors with products for algebraic schemes. Here ``oriented" means that the given Borel--Moore functor $H_*$ is equipped with an endomorphsim $\tilde c_1(L): H_*(X) \to H_*(X)$ for a line bundle $L$ over the scheme $X$. Motivated by this ``orientation", we introduce an orientation to bivariant theories for any category, using the notion of \emph {fibered categories} in abstract category theory, e.g, see \cite{Vistoli}. 

\begin{defn}
Let $\Cal L$ be a fibered category over $\Cal V$. An object in the fiber $\Cal L(X)$ over an object $X \in \Cal V$ is called an \emph {``fiber-object over $X$"}, abusing words, and denoted by $L$, $M$, etc.
\end{defn}

\begin{defn}\label{orientation} Let $\bB$ be a bivariant theory on a category $\Cal V$ and let $\Cal L$ be a fibered category over $\Cal V$. For a fiber-object $L$ over $X$, the \emph{``operator" on $\bB$ associated to the fiber-object $L$}, denoted by $\phi(L)$, is defined to be an \emph{endomorphism}
$$\phi(L): \bB(X  \xrightarrow{f}  Y) \to \bB(X  \xrightarrow{f}  Y) $$
which satisfies the following properties:

(O-1) {\bf identity}: If $L$ and $L'$ are fiber-objects over $X$ and isomorphic, then we have
$$\phi(L) = \phi(L'): \bB(X  \xrightarrow{f}  Y) \to \bB(X  \xrightarrow{f}  Y).$$

(O-2) {\bf commutativity}: Let $L$ and $L'$ be two fiber-objects over $X$, then we have
$$\phi(L) \circ \phi(L') = \phi(L') \circ \phi(L) :\bB(X  \xrightarrow{f}  Y) \to \bB(X  \xrightarrow{f}  Y). $$

(O-3)  {\bf compatibility with product}: For morphisms $f:X \to Y$ and $g:Y \to Z$,  $\alp \in \bB(X  \xrightarrow{f} Y)$ and $ \be \in \bB(Y  \xrightarrow{g} Z)$,  a fiber-object $L$ over $X$ and a fiber-object $M$ over $Y$
 $$ \phi(L) (\alp \bullet \be) = \phi(L)(\alp) \bullet \be, \quad  \phi(f^*M) (\alp \bullet \be) = \alp \bullet \phi(M)(\be)$$

(O-4)  {\bf compatibility with pushforward}: For a confined morphism $f:X \to Y$ and a fiber-object $M$ over $Y$ 
$$ f_*\left (\phi(f^*M)(\alp) \right ) = \phi(M)(f_*\alp).$$

(O-5)   {\bf compatibility with pullback}: For an independent square and a fiber-object $L$ over $X$
$$
\CD
X' @> g' >> X \\
@V f' VV @VV f V\\
Y' @>> g > Y 
 \endCD
$$

$$g^*\left (\phi(L)(\alp) \right ) = \phi({g'}^*L)(g^*\alp).$$

The above operator is called an ``{\it orientation}" and a bivariant theory equipped with such an orientation is called an {\it oriented bivariant theory}, denoted by $\bOB$. An {\it oriented Grothendieck transformation} between two oriented bivariant theories is a Grothendieck transformation which preserves or is compatible with the operator, i.e., for two oriented bivariant theories $\bOB$ with an orientation $\phi$ and $\bOB'$ with an orientation $\phi'$ the following diagram commutes
$$
\CD
\bOB (X  \xrightarrow{f}  Y)  @> {\phi(L)}>> \bOB (X  \xrightarrow{f}  Y) \\
@V \ga VV @VV \ga V\\
\bOB' (X  \xrightarrow{f}  Y) @>>{\phi'(L)} > \bOB' (X  \xrightarrow{f}  Y).
 \endCD
$$
\end{defn} 

\begin{rem} All we need above is only the fact that  it is `` closed under pull-back" or ``closed under base change". Thus, in this sense, we can define the above operator for a certain class $\Cal L$ of morphisms which is closed under base change; i.e., $f: L \to X \in \Cal L$ if and only if for any morphism $g: X' \to X \in \Cal V$ and the fiber square
$$
\CD
L' @> g' >> L \\
@V f' VV @VV f V\\
X' @>> g > X,
 \endCD
$$
the pull-back $f':L' \to X'$  belongs to $\Cal L$. Originally we considered this situation, however we delt with more generally fibered categories (suggested by J\"org Sch\"urmann). \\
\end{rem}

The following lemma shows that Levine--Morel's {\it oriented Borel--Moore functor with products} is a special case of an oriented bivariant theory.

\begin{lem} Let $\bOB$ be an oriented bivariant theory on a category $\Cal V$ with $\Cal L$ a fibered category over $\Cal V$. Then the orientation $\phi$ on the functors $\bOB_*$ and $\bOB^*$ satisfies the following properties:

(1) Let $L$ and $L'$ be two fiber-objects over $X$, then we have
$$\phi (L) \circ \phi(L') = \phi(L') \circ \phi(L) :\bOB_*(X) \to \bOB_*(X),$$
$$\phi (L) \circ \phi(L') = \phi(L') \circ \phi(L) :\bOB^*(X) \to \bOB^*(X),$$
and if $L$ and $L'$ are isomorphic, then we have that $\phi(L) = \phi (L')$ for both $\bOB_*$ and $\bOB^*$.

(2) For a fiber-object $L$ and $\alp \in \bOB_*(X)$ and $\be \in \bOB_*(Y)$, we have
 $$\phi(L) (\alp) \times \be = \phi({p_1}^*L)(\alp \times \be) .$$
Also, for $\alp \in \bOB^*(X)$ and $\be \in \bOB^*(Y)$, we have
 $$\phi(L) (\alp) \times \be = \phi({p_1}^*L)(\alp \times \be) .$$
Here $p_1: X \times Y \to X$ is the projection.

(3) For a confined morphism $f: X \to Y$ and a fiber-object $M$ over $Y$, we have
$$ f_* \circ \phi(f^*M) = \phi(M) \circ f_*: \bOB_*(X) \to \bOB_*(Y).$$

(4) For a specialized morphism $f :X \to Y \in \Cal S$ (here we just require that $f$ is canonically $\bOB$-oriented) and a fiber-object $M$ over $Y$, we have
$$\phi (f^*M) \circ f^! = f^! \circ \phi (M): \bOB_*(Y) \to \bOB_*(X).$$

(5) For a confined and specialized morphism $f: X \to Y$ and a fiber-object $M$ over $Y$, we have
$$f_! \circ \phi (f^*M) = \phi (M) \circ f_!: \bOB^*(X) \to \bOB^*(Y).$$

(6) For any morphism $f: X \to Y$ and a fiber-object $M$ over $Y$, we have
$$ \phi (f^*M) \circ f^* = f^*\circ \phi(M) : \bOB^*(Y) \to \bOB^*(X).$$

\end{lem}

\begin{proof}
(1) follows from (O-1) and (O-2). 

(2) follows from the first formula of (O-3). 

(3) follows from (O-4).

 (4) follows from the second formula of (O-3). 

(5) follows from the first formula of (O-3) and (O-4).

(6) follows from (O-5). 
\end{proof}

\begin{ex} All the examples, except the bivariant theory $\bF$ of constructible functions, given in Example (\ref{BT-example}) are in fact oriented bivariant theories, if we consider line bundles (or any bundles) for a fibered category and Chern classes (or any characteristic classes) for operators. The operator  $\phi(L):= c_1(L) \bullet $  is taking the bivariant product with the first Chern class as a bivariant element $c_1(L)$ in the bivariant group of the identity $X \xrightarrow{\op {id}_X} X$.  One can of course consider another operator $\phi'(L) := c(L) \bullet$, the bivariant product with the total Chern class of the line bundle.  Then it follows from the axioms of the bivariant theory that this operator satisfies the properties (O-2) --- (O-5).  For example, in the case of bivariant homology theory $\bH$:
For a line bundle $L \to X$, the first Chern class operator $\tilde c_1(L) :\bH(X  \xrightarrow{f} Y) \to \bH(X  \xrightarrow{f} Y) $
is defined by $\tilde c_1(L)(\alp) := c_1(L) \bullet \alp$, 
where $c_1(L) \in \bH(X  \xrightarrow{\op {id}_X} X) = H^*(X)$ is the first Chern cohomology class of the line bundle.  However, as to the bivariant theory $\bF$ of constructible functions, $\bF^*(X) = \bF(X  \xrightarrow{\op {id}_X}  X)$ consists of locally constant functions. So, for a vector bundle $E$ over $X$, we do not know any reasonable geometrically or topologically defined operator $\phi(E): \bF(X  \xrightarrow{\op {id}_X}  X) \to \bF(X  \xrightarrow{\op {id}_X}  X)$ other than the multiplication of the rank of the vector bundle $E$. 
\end{ex}

In fact, mimicking Levine--Morel's construction \cite{LM-book}, we show the existence of a universal one among such oriented bivariant theories for \emph {any} category $\Cal V$ and a fibered category $\Cal L$ over the category $\Cal V$ .

Let us consider a morphism $h_V:V \to X$ \emph {equipped with} finitely many fiber-objects  over the source variety $V$ of the morphism $h_V$:
$$(V  \xrightarrow{h_V} X; L_1, L_2, \cdots, L_r)$$
with $L_i$ being a fiber-object over $V$. Note of course that we can consider only the morphism $(V  \xrightarrow{h_V} X)$ without any fiber-objects equipped. This family is called 
a {\it cobordism cycle over $X$}, following \cite{LM-book} (see also \cite{Merkurjev}). Then 
$(V  \xrightarrow{h_V} X; L_1, L_2, \cdots, L_{r'})$ is defined to be isomorphic to  $(W  \xrightarrow{h_W} X; M_1, M_2, \cdots, M_r)$ if and only if $h_V$ and $h_W$ are isomorphic, i.e., there is an isomorphism $g:V \cong W$ over $X$, there is a bijection $\sigma: \{1, 2, \cdots, r \}\cong  \{1, 2, \cdots, r'\}$ (so that $r = r'$) and there are isomorphisms $L_i \cong g^*M_{\sigma(i)}$ for every $i$.

\begin{thm} (A universal oriented bivariant theory)
 Let  $\Cal V$ be a cateogry with a class $\Cal C$ of confined morphisms, a class of independent squares, a class  $\Cal S$ of specialized morphisms and a fibered category $\Cal L$ over $\Cal V$.  We define 
$$\bOM^{\Cal C} _{\Cal S}(X  \xrightarrow{f}  Y)$$
to be the free abelian group generated by the set of isomorphism classes of cobordism cycles over $X$
$$[V  \xrightarrow{h} X; L_1, L_2, \cdots, L_r]$$
such that the composite of $h$ and $f$
$$f \circ h: W \to Y \in \Cal S.$$

(1) The assignment $\bOM^{\Cal C} _{\Cal S}$ becomes an oriented bivariant theory if the four operations are defined as follows:

{\bf Orientation $\Phi$}: For a morphism $f:X \to Y$ and a fiber-object $L$ over $X$, the operator
$$\Phi (L):\bOM^{\Cal C} _{\Cal S} ( X  \xrightarrow{f}  Y) \to \bOM^{\Cal C} _{\Cal S} ( X  \xrightarrow{f}  Y) $$
is defined by
$$\Phi(L)([V  \xrightarrow{h_V} X; L_1, L_2, \cdots, L_r]):=[V  \xrightarrow{h_V} X; L_1, L_2, \cdots, L_r, (h_V)^*L].$$

{\bf Product operations}: For morphisms $f: X \to Y$ and $g: Y
\to Z$, the product operation
$$\bullet: \bOM^{\Cal C} _{\Cal S} ( X  \xrightarrow{f}  Y) \otimes \bOM^{\Cal C} _{\Cal S} ( Y  \xrightarrow{g}  Z) \to
\bOM^{\Cal C} _{\Cal S} ( X  \xrightarrow{gf}  Z)$$
is  defined as follows: The product on generators is defined by
\begin{align*}
& [V  \xrightarrow{h_V} X; L_1, \cdots, L_r]  \bullet [W  \xrightarrow{k_W} Y; M_1, \cdots, M_s] \\
& :=  [V'  \xrightarrow{h_V \circ k''_W}  X; {k''_W}^*L_1, \cdots,{k''_W}^*L_r, (f' \circ {h'_V})^*M_1, \cdots, (f' \circ {h'_V})^*M_s ],\
\end{align*}
and it extends bilinearly. Here we consider the following fiber squares
$$\CD
V' @> {h'_V} >> X' @> {f'} >> W \\
@V {k''_W}VV @V {k'_W}VV @V {k_W}VV\\
V@>> {h_V} > X @>> {f} > Y @>> {g} > Z .\endCD
$$

{\bf Pushforward operations}: For morphisms $f: X \to Y$
and $g: Y \to Z$ with $f$ confined, the pushforward operation
$$f_*: \bOM^{\Cal C} _{\Cal S} ( X  \xrightarrow{gf} Z) \to \bOM^{\Cal C} _{\Cal S} ( Y  \xrightarrow{g}  Z) $$
is  defined by
$$f_*\left (\sum_Vn_V[V  \xrightarrow{h_V} X; L_1, \cdots, L_r]  \right) := \sum _Vn_V[V  \xrightarrow{f \circ h_V} Y; L_1, \cdots, L_r] .$$

{\bf  Pullback operations}: For an independent square
$$\CD
X' @> g' >> X \\
@V f' VV @VV f V\\
Y' @>> g > Y, \endCD
$$
the pullback operation
$$g^*: \bOM^{\Cal C} _{\Cal S} ( X  \xrightarrow{f} Y) \to \bOM^{\Cal C} _{\Cal S}( X'  \xrightarrow{f'} Y') $$
is  defined by
$$g^*\left (\sum_V n_V[V  \xrightarrow{h_V} X; L_1, \cdots, L_r] \right):=  \sum_V n_V[V'  \xrightarrow{{h'_V}}  X'; {g''}^*L_1, \cdots, {g''}^*L_r],$$
where we consider the following fiber squares:
$$\CD
V' @> g'' >> V \\
@V {h_V'} VV @VV {h_V}V\\
X' @> g' >> X \\
@V f' VV @VV f V\\
Y' @>> g > Y. \endCD
$$

(2) Let $\Cal {OBT}$ be a class of oriented bivariant theories $\bOB$ on the same category $\Cal V$ with a class $\Cal C$ of confined morphisms, a class of independent squares, a class $\Cal S$ of specialized morphisms and a fibered category $\Cal L$ over $\Cal V$. Let $\Cal S$ be nice canonically $\bOB$-oriented for any oriented bivariant theory $\bOB \in \Cal {OBT}$. Then, for each oriented bivariant theory $\bOB \in \Cal {OBT}$ with an orientation $\phi$  there exists a unique oriented Grothendieck transformation
$$\ga_{\bOB} : \bOM^{\Cal C} _{\Cal S} \to \bOB$$
such that for any $f: X \to Y \in \Cal S$ the homomorphism
$\ga_{\bOB} : \bOM^{\Cal C} _{\Cal S}(X  \xrightarrow{f}  Y) \to \bOB(X  \xrightarrow{f}  Y)$
satisfies the normalization condition that $$\ga_{\bOB}([X  \xrightarrow{\op {id}_X}  X; L_1, \cdots, L_r]) = \phi(L_1) \circ \cdots \circ \phi(L_r) (\theta_{\bOB}(f)).$$
\end{thm}

\begin{proof} (1): It is easy to see that the above four operations are well-defined.  Here we also make the following observations:

\noindent 
\underline{Observation (*)}: Let $L$ be a fiber-object over $X$. For $[X \xrightarrow{\op {id}_X}  X; L] \in \bOM^{\Cal C} _{\Cal S}(X \xrightarrow{\op {id}_X}  X)$ and  $[V  \xrightarrow{h_V} X; L_1, L_2, \cdots, L_r] \in \bOM^{\Cal C} _{\Cal S}(X \xrightarrow{f}  Y)$, by the definition of bivariant product and by using the following fiber squares 

$$\CD
V @> {id_V} >> V @> {id_V} >> V \\
@V {h_V} VV @V {h_V} VV @V {h_V}VV \\
X@>> {id_X} > X @>> {id_X} > X @>> {f} > Y \endCD
$$
we have
\begin{align*}
[X \xrightarrow{\op {id}_X}  X; L] \bullet & [V  \xrightarrow{h_V} X; L_1, L_2, \cdots, L_r] \\
& =  [V  \xrightarrow{h_V} X; (h_V)^* L, L_1, L_2, \cdots, L_r] \\
& =  [V  \xrightarrow{h_V} X;  L_1, L_2, \cdots, L_r, (h_V)^* L] \\
& = \Phi(L) ([V  \xrightarrow{h_V} X;  L_1, L_2, \cdots, L_r]) 
\end{align*}
Hence the above operator  $\Phi (L):\bOM^{\Cal C} _{\Cal S} ( X  \xrightarrow{f}  Y) \to \bOM^{\Cal C} _{\Cal S} ( X  \xrightarrow{f}  Y) $ is the same as $[X \xrightarrow{\op {id}_X}  X; L] \bullet$, i.e., taking the bivariant product with the ``motivic" class of $L$, $[X \xrightarrow{\op {id}_X}  X; L] \in \bOM^{\Cal C} _{\Cal S}(X \xrightarrow{\op {id}_X}  X)$.

\noindent
\underline{Observation (**)}: For $[V  \xrightarrow{h_V} X; L_i] \in \bOM^{\Cal C} _{\Cal S}(X \xrightarrow{f}  Y)$ with $L_i$ a fiber-object over $V$, we have $[V  \xrightarrow{id_V} V; L_i] \in \bOM^{\Cal C} _{\Cal S}(V \xrightarrow{f \circ h_V}  Y)$ and 
$$[V  \xrightarrow{h_V} X; L_i] = {h_V}_* [V  \xrightarrow{id_V} V; L_i].$$
In general, $[V  \xrightarrow{h_V} X; L_1, \cdots, L_r] = {h_V}_* [V  \xrightarrow{id_V} V;L_1, \cdots, L_r].$ Furthermore we have 
$$[V  \xrightarrow{id_V} V; L_i] = \Phi(L_i)([V  \xrightarrow{id_V} V]).$$
 In general, $[V  \xrightarrow{id_V} V; L_1, \cdots, L_r] = \Phi(L_1) \circ \cdots \Phi(L_r) ([V  \xrightarrow{id_V} V])$. Therefore we get that
$$[V  \xrightarrow{h_V}  X; L_1, \cdots, L_r] =  {h_V}_* \left (\Phi (L_1) \circ \cdots \circ \Phi (L_r) ([V  \xrightarrow{\op {id}_V}  V]) \right ).$$

(2): Suppose that there is an oriented Grothendieck transformation
$$\ga: \bOM^{\Cal C} _{\Cal S} \to \bOB$$
satisfying that for any $f: X \to Y \in \Cal S$ the homomorphism
$\ga: \bOM^{\Cal C} _{\Cal S} (X  \xrightarrow{f}  Y) \to \bOB(X  \xrightarrow{f}  Y)$
satisfies that $\ga([X  \xrightarrow{\op {id}_X}  X]) = \theta_{\bOB}(f).$  It suffices to show that the value of any generator 
$$[V  \xrightarrow{h_V}  X; L_1, \cdots, L_r] \in \bOM^{\Cal C} _{\Cal S} (X  \xrightarrow{f}  Y)$$
is uniquely determined.  From the above Observation (**), we have

\begin{align*}
\ga([V  \xrightarrow{h_V}  X; L_1, \cdots, L_r] ) 
& = \ga \left ( {h_V}_* \left (\Phi (L_1) \circ \cdots \circ \Phi (L_r) ([V  \xrightarrow{\op {id}_V}  V]) \right ) \right ) \\
& = {h_V}_* \left (\ga  \left (\Phi (L_1) \circ \cdots \circ \Phi (L_r) ([V  \xrightarrow{\op {id}_V}  V]) \right ) \right ) \\
&= {h_V}_*\left (\phi (L_1) \circ \cdots \circ \phi(L_r) \ga([V  \xrightarrow{\op {id}_V}  V]) \right ) \\
&= {h_V}_* \left (\phi(L_1) \circ \cdots \circ \phi (L_r) \theta_{\bOB}(f \circ h_V) \right ) \
\end{align*}

Thus the uniqueness follows. 

Next, we show the existence of such an oriented Grothendieck transformation satisfying the above normalization condition. We define the assignment
$$\ga_{\bOB} : \bOM^{\Cal C} _{\Cal S}(X  \xrightarrow{f}  Y) \to \bOB(X  \xrightarrow{f}  Y)$$
by
$$\ga_{\bOB} ([V  \xrightarrow{h_V}  X; L_1, \cdots, L_r] ) 
:= {h_V}_* \left (\phi(L_1) \circ \cdots \circ \phi(L_r) \theta_{\bOB}(f \circ h_V) \right ). $$
This certainly satisfies the normalization condition.

 The rest is to show that it is an oriented Grothendieck transformation.
 
(i) \underline {it preserves the product operation}:It suffices to show that
\begin {align*}
& \ga_{\bOB} \left ([V  \xrightarrow{h_V} X;L_1, \cdots, L_r ]\bullet [W  \xrightarrow{k_W} Y; M_1, \cdots, M_s] \right )  \\
& = \ga_{\bOB}([V  \xrightarrow{h_V} X; L_1, \cdots, L_r]) \bullet \ga_{\bOB}([W  \xrightarrow{k_W} Y;M_1, \cdots, M_s])\
\end{align*}
 Using some parts of the proof of Theorem (\ref{UBT}), we have 
\begin {align*}
& \ga_{\bOB} \left ([V  \xrightarrow{h_V} X;L_1, \cdots, L_r ]\bullet [W  \xrightarrow{k_W} Y; M_1, \cdots, M_s] \right )  \\
& =  \ga_{\bOB} \bigl ([V'  \xrightarrow {h_V \circ k''_W}  X; {k''_W}^*L_1, \cdots,{k''_W}^*L_r, (f' \circ h'_V)^*M_1, \cdots, (f' \circ h'_V)^*M_s ]\bigr )  \\
& = {h_V}_*{k''_W}_* \bigl (\phi({k''_W}^*L_1) \circ \cdots \circ \phi({k''_W}^*L_r) \circ \\
&  \hspace{1cm} \phi((f' \circ h_V')^*M_1) \circ \cdots 
\circ \phi((f' \circ h_V')^*M_s) \bigr ) ( \theta_{\bOB}(f' \circ h_V') \bullet \theta_{\bOB}(g \circ k_W)) \bigr ) . \
\end{align*}

Here we use the property (O-4)  {\bf compatibility with pushforward} and (O-3)  {\bf compatibility}  {\bf  with product} [$\phi(f^*M) (\alp \bullet \be) = \alp \bullet \phi(M)(\be)$],   the above equality continues as follows:

\begin{align*}
& = {h_V}_*\bigl (\phi(L_1) \circ \cdots \circ \phi(L_r)  \circ { k''_W}_* \\
& \hspace {1cm} \left (\phi((f' \circ h_V')^*M_1) \circ \cdots 
\circ \phi((f' \circ h_V')^*M_s) \right ) ( \theta_{\bOB}(f' \circ h_V') \bullet \theta_{\bOB}(g \circ k_W) 
\bigr ). \\
& ={h_V}_* \left (\phi(L_1) \circ \cdots \circ \phi(L_r) \right)  \circ ( k''_W)_* \\
& \hspace{3cm} \left (\theta_{\bOB}(f' \circ h_V') \bullet  \left (\phi(M_1) \circ \cdots 
\circ \phi(M_s) \right ) \theta_{\bOB}(g \circ k_W) \right ). \\
& = {h_V}_*\bigl (\phi(L_1) \circ \cdots \circ \phi(L_r)  \circ { k''_W}_* \\
& \hspace{3cm} \left ((k_W)^*\theta_{\bOB}(f \circ h_V) \bullet  (\phi(M_1) \circ \cdots 
\circ \phi(M_s) \right ) \theta_{\bOB}(g \circ k_W) \bigr ). \\
& = {h_V}_*\left (\phi(L_1) \circ \cdots \circ \phi(L_r) \right) \\
& \hspace{3cm} \left (\theta_{\bOB}(f \circ h_V) \bullet  {k_W}_*\left (\phi(M_1) \circ \cdots 
\circ \phi(M_s) \right ) \theta_{\bOB}(g \circ k_W) \right ) \\
& \hspace{8cm}  \text {(by (B-7) projection formula).} 
\end {align*}

Furthermore, using (O-3)  {\bf compatibility with product} [$ \phi(L) (\alp \bullet \be) = \phi(L)(\alp) \bullet \be$ ] and by (B-4), it continues as follows:

\begin{align*}
& = {h_V}_*\left (\phi(L_1) \circ \cdots \circ \phi(L_r) \theta_{\bOB}(f \circ h_V) \right) \bullet  \\
& \hskip 3cm {k_W}_*\left (\phi(M_1) \circ \cdots \circ \phi(M_s)  \theta_{\bOB}(g \circ k_W) \right) \\
& = \ga_{\bOB}([V  \xrightarrow{h_V} X; L_1, \cdots, L_r]) \bullet \ga_{\bOB}([W  \xrightarrow{k_W} Y;M_1, \cdots, M_s])\
\end{align*}

(ii) \underline {it preserves the pushforward operation}:Consider $X  \xrightarrow{f} Y \xrightarrow{g} Z$ and a confined morphsim $h_V:V \to X$ such that the composite $(g \circ f )\circ h_V: V \to Y$ is in $\Cal S$. For a generator
$$[V  \xrightarrow{h_V} X;L_1, \cdots, L_r] \in  \bOM^{\Cal C} _{\Cal S}(X  \xrightarrow{g \circ f}  Z),$$
we have
\begin{align*}
& \ga_{\bOB}(f_*[V  \xrightarrow{h_V} X;L_1, \cdots, L_r]) \\
&= \ga_{\bOB}([V  \xrightarrow{f \circ h_V} Y;L_1, \cdots, L_r]) \\
& = (f \circ h_V)_*\left (\phi(L_1) \circ \cdots \circ \phi(L_r) (\theta_{\bOB}(g \circ (f \circ h_V))\right) \\
& = f_*{h_V}_*\left (\phi(L_1) \circ \cdots \circ \phi(L_r) \theta_{\bOB}((g \circ f) \circ h_V) \right) \\
& = f_* \ga_{\bOB}([V  \xrightarrow{h_V} X;L_1, \cdots, L_r]).\
\end{align*}

(iii) \underline {it preserves the pullback operation}: Consider a confined morphsim $h_V:V \to X$ such that the composite $f \circ h_V: V \to Y$ is in $\Cal S$ and the fiber squares given in {\bf Pullback operations} above, we have

\begin{align*}
& \ga_{\bOB}(g^*[V  \xrightarrow{h_V} X;L_1, \cdots, L_r]) \\&= \ga_{\bOB}([V'  \xrightarrow{
h_V'} X'; {g''}^*L_1, \cdots, {g''}^*L_r]) \\
& = {h_V'}_*\left (\phi({g''}^*L_1) \circ \cdots \circ \phi({g''}^*L_r) \theta_{\bOB}(f' \circ h_V') \right)\\
& = {h_V'}_*\left (\phi({g''}^*L_1) \circ \cdots \circ \phi({g''}^*L_r)  g^*\theta_{\bOB}(f \circ h_V) \right) \\
& = {h_V'}_*g^*\left (\phi(L_1) \circ \cdots \circ \phi(L_r) \theta_{\bOB}(f \circ h_V) \right)\\
& = g^*{h_V}_*\left (\phi(L_1) \circ \cdots \circ \phi(L_r) \theta_{\bOB}(f \circ h_V) \right)\quad \text {(by (B-5))}\\
& = g^*\ga_{\bOB}([V  \xrightarrow{h_V} X;L_1, \cdots, L_r]).\
\end{align*}
\end{proof}

\begin{cor} The abelian group ${\bOM^{\Cal C} _{\Cal S}}_*(X):= {\bOM^{\Cal C} _{\Cal S}}(X \to pt)$
is the free abelian group generated by the set of  isomorphism classes of cobordism cycles
$$[V  \xrightarrow{h_V} X;L_1, \cdots, L_r]$$
such that $h_V:V \to X \in \Cal C$ and $V \to pt$ is a specialized map in $\Cal S$ and $L_i$ is a fiber-object over $V$. 
The abelian group ${\bOM^{\Cal C} _{\Cal S}}^*(X):= {\bOM^{\Cal C} _{\Cal S}}(X  \xrightarrow{\op {id}_X} X)$
is the free abelian group generated by the set of  isomorphism classes of cobordism cycles
$$[V  \xrightarrow{h_V} X;L_1, \cdots, L_r]$$
such that $h_V:V \to X \in \Cal C \cap \Cal S$ and $L_i$ is a fiber-object over $V$. Both functor ${\bOM^{\Cal C} _{\Cal S}}_*$ and ${\bOM^{\Cal C} _{\Cal S}}^*$  are oriented Borel--Moore functors with products in the sense of Levine--Morel.\\
\end{cor}

\begin{cor} (A universal oriented Borel--Moore functor with products) 
Let $\Cal {BT}$ be a class of oriented additive bivariant theories $\bB$ on the same category $\Cal V$ with a class $\Cal C$ of confined morphisms, a class of independent squares, a class $\Cal S$ of specialized maps and a fibered category $\Cal L$ over $\Cal V$. Let $\Cal S$ be nice canonically $\bOB$-oriented for any oriented bivariant theory $\bOB \in \Cal {OBT}$. Then, for each oriented bivariant theory $\bOB \in \Cal {OBT}$ with an orientation $\phi$,

(1) there exists a unique natural transformation of oriented Borel--Moore functors with products
$${\gamma _{\bOB}}_* : {\bOM^{\Cal C} _{\Cal S}}_* \to \bOB_*$$
such that if $\pi_X:X \to pt$ is in  $\Cal S$ 
$${\gamma _{\bOB}}_* [X  \xrightarrow{\op {id}_X} X; L_1, \cdots, L_r] = \phi (L) \circ \cdots \circ \phi (L_r) ({\pi_X}^!(1_{pt})),$$

and 

(2) there exists a unique natural transformation of oriented Borel--Moore functors with products
$${\gamma _{\bOB}}^* : {\bOM^{\Cal C} _{\Cal S}}^* \to \bOB^*$$
such that for any object $X$ 
$${\gamma _{\bOB}}^* [X  \xrightarrow{\op {id}_X} X; L_1, \cdots, L_r] = \phi (L) \circ \cdots \circ \phi (L_r) (1_X).$$ \\
\end{cor}

\begin{rem} (1) Let $k$ be an arbitrary field. In the case when $\Cal V_k$ is the admissible subcategory of the category of separated schemes of finite type over the field $k$, $\Cal C = \Cal Proj$ is the class of projective morphisms, $\Cal S = \Cal Sm$ is the class of smooth equi-dimensional morphisms and $\Cal L$ is the class of line bundles, then ${\bOM_{\Cal Sm}^{\Cal Proj}}_*(X) = \bOM_{\Cal Sm}^{\Cal Proj}(X \to {pt}) $
is nothing but the oriented Borel--Moore functor with products $\Cal Z_*(X)$ given in \cite{LM-book}. In this sense, our associated contravariant one ${\bOM_{\Cal Sm}^{\Cal Proj}}^*(X) = \bOM_{\Cal Sm}^{\Cal Proj}(X  \xrightarrow{\op {id}_X}  X) $
is a ``cohomological" counterpart of Levine--Morel's ``homological one" $\Cal Z_*(X)$. Note that this cohomological one ${\bOM_{\Cal Sm}^{\Cal Proj}}^*(X)$ for \emph {any} scheme $X$ is the free abelian group generated by $[V  \xrightarrow{h_V} X; L_1, \cdots, L_r]$ such that $h_V:V \to X$ is a \emph {projective and smooth} morphism.

(2) One can see that in (1) $\Cal Proj$ can be replaced by $\Cal Prop$. And furthermore one can consider 
$\bOM_{\Cal Lci}^{\Cal Proj}$ and $\bOM_{\Cal Lci}^{\Cal Prop}$, which will be  treated in \cite{Yokura-bac}. Here $\Cal Lci$ is the class of local complete intersection morphisms.
\end{rem}

An oriented bivariant theory can be defined for any kind of category as long as we can specify classes of ``confined morphisms", ``specialized morphisms" together with nice canonical orientations, ``independent squares" and a ``fibered category" over the given category. The above oriented bivariant theory is the very basis of other oriented bivariant theories of more geometric natures. In  \cite{Yokura-bac} we will deal with a more geometrical oriented bivariant theory, i.e., what could  be called \emph{bivariant algebraic cobordism} or {\it algebraic bivariant cobordism}, which is a bivariant version of Levine--Morel's algebraic cobordism. \\

\section*{Acknowledgments} 

I would like to thank the staff of the  Erwin Schr\"odinger International Institute of Mathematical Physics (ESI) for their hospitality and for a nice atmosphere in which to  work during my stay in August, 2006, during which the present research was initiated. I would also like to thank J\"org Sch\"urmann for his useful comments and suggestions on the earlier version \cite {Yokura-obt1} of the paper. Finally I would like to thank the referee for his/her comments. 


\end{document}